\documentclass[a4paper,11pt,fleqn,reqno]{amsart}

\usepackage{amsmath,amsfonts,amsthm,amssymb,amsxtra,enumerate,color,graphicx,times}

\newcommand{\C}{\mathbb C}

\newcommand{\PP}{\mathbb P}
\newcommand{\supess}{\mathop{\rm supess}\nolimits}

\renewcommand{\Re}{\mathop {\rm Re}\nolimits}

 \newtheorem{theorem}{Theorem}
 \newtheorem{corollary}[theorem]{Corollary}
 \newtheorem{lemma}[theorem]{Lemma}
 
 \newtheorem{remark}[theorem]{Remark}

\begin{document}

\title{Faber polynomials of matrices for non-convex sets}

\dedicatory{Pour Paul Sablonni\`ere, \`a l'occasion de son soixante-cinqui\`eme anniversaire}

\author{B. Beckermann \& M. Crouzeix}

\date{4/10/13}

\begin{abstract}
    It has been recently shown that $|| F_n(A) ||\leq 2$,
  where $A$ is a linear continuous operator acting in a
  Hilbert space,
  and $F_n$ is the Faber polynomial of degree $n$ corresponding
  to some convex compact $E\subset \mathbb C$ containing the
  numerical range of $A$. Such an inequality is useful in numerical linear algebra,
  it allows for instance to derive error bounds
  for Krylov subspace methods.
  In the present paper we extend this result to not necessary
  convex sets $E$.
\end{abstract}

\maketitle

\bigskip


\noindent {\bf Classification AMS:}   47A12, 65F10

\section{Introduction and statement of the main results}

Consider a bounded operator $A$ on a Hilbert space $\mathcal H$ with spectrum $\sigma(A)$, for example a square matrix $A\in\mathbb C^{N \times N}$, and denote by $\PP_n$ the space of polynomials of degree $\leq n$.
Following \cite{grtr,tt2}, we are interested in giving upper bounds for the quantity
\begin{align}\label{but}
      \delta_n(A)=\min\{\|p(A)\|\,; p\in \PP_{n}, \ p(0)=1 \} , \quad n =1,2,3,...,
\end{align}
sometimes referred to as the ideal GMRES approximation problem \cite{grtr,tt1}.
For normal $A$, problem \eqref{but} is closely related to the so-called constrained Chebyshev approximation problem
 \begin{align}\label{def2}
      \delta (n,E)=\min\{\|p\|_{L^\infty(E)}\,; p\in \PP_{n}, \ p(0)=1 \},
\end{align}
for a suitable compact $E\subset \mathbb C$ not containing $0$ (since otherwise $\delta(n,E)=1$).
This latter quantity has been discussed for intervals/ellipses $E$ for instance in \cite{fisch1,fisch2,fisch4,freund}, see also the monograph \cite{fisch3}.
Though in general it is difficult to find extremal polynomials for \eqref{def2}, we can find nearly optimal ones.

Given a compact $E \subset \mathbb C$ with a rectifiable Jordan curve boundary, we define as usual the $n$th Faber polynomial
$F_n=F_n^E$ to be the polynomial part of the Laurent expansion at
infinity of $\Phi$, where the Riemann map $\Phi$ maps
conformally the exterior of $E$ onto the exterior of the closed
unit disk $\mathbb D$, and $\Phi(\infty)=\infty$,
$\Phi'(\infty)>0$. Thus,
 \begin{align*}
      F_n(z)=\Phi(z)^n+O(1/z),\quad \hbox{as  }z\to\infty.
\end{align*}

We also introduce the geometric quantity
\begin{align}\label{geometric}
      v(E)=\supess\{\frac1\pi \int_{\partial E_z}|d_\sigma \arg(\sigma-z)|\,; z\in \partial E\}, \quad\hbox{where  }\partial E_z=\partial E\setminus\{z\},
\end{align}
which we assume to be finite. Note that $v(E)\geq 1$, and $v(E)=1$ if $E$ is convex. An estimate due to Radon \cite{rad} tells us that $1{+}v(E)\leq TV(\theta)/2\pi $, where $TV(\theta)$ denotes the total variation of the angle $\theta$ between the positive real axis and the tangent on the boundary $\partial E$.
The following properties have essentially been given established by K\"ovari and Pommerenke in \cite{koepo,pom}.

\begin{lemma}\label{lem_faber}
   Let $E$ be as above, $0\not\in E$, $\gamma=1/|\Phi(0)|$, then $\| F_n \|_{L^\infty(E)}\leq 1+v(E)$, and
  \begin{align}\label{value_at_zero}
      \frac{1}{|F_n(0)|} \leq \frac{\gamma^n}{1-\gamma^{n+1} v(E)}
  \end{align}
  provided that $\gamma^{n+1} v(E)<1$.
\end{lemma}
For the sake of completeness, we will give in \S 2 a proof of this statement. From Lemma~\ref{lem_faber} and the maximum principle for $p/\Phi^n$, $p \in \PP_n$, we conclude that, provided that $\gamma^{n+1} v(E)<1$
\begin{align*}
      \gamma ^n\leq \delta (n,E)\leq \frac{\| F_n \|_{L^\infty(E)}}{|F_n(0)|} \leq \gamma^n \frac{1+v(E)}{1-\gamma ^{n+1}v(E)}.
\end{align*}
One attempt to relate this inequality to the matrix-valued extremal problem \eqref{but} could be to make use of the notion of $K$-spectral sets, see for instance \cite{babe} and the references therein: we look for $E\subset \mathbb C$ containing $\sigma(A)$ and a constant $K>0$ such that $\| p(A) \|\leq K \ \| p \|_{L^\infty(E)}$ for all $p\in \mathbb P_n$, and thus $\delta_n(A)\leq K \, \delta(n,E)$. For instance, if
$X^{-1} A X$ is normal, then $K=\| X \| \, \| X^{-1} \|$. Natural candidates for $E$ are obtained by the pseudo-spectrum,  or the numerical range being defined by
\begin{align*}
    W(A)=\{\langle Au,u\rangle\,; u\in \mathcal H, \|u\|=1\},
\end{align*}
see for instance the discussion in \cite{greenbaum} and \cite{crzx,crou4}.
In the present paper we will use instead the inequality
\begin{align}\label{est6bis}
      \delta_n (A)\leq \frac{\| F_n(A) \|}{|F_n(0)|} \leq \gamma^n \frac{\| F_n(A) \|}{1-\gamma ^{n+1}v(E)}
\end{align}
being a consequence of \eqref{value_at_zero}, where it remains the simpler task of
bounding $\| F_n(A) \|$ for a suitable $E\subset \mathbb C$. Previous work on this subject includes \cite{tt2} where the bound of Kreiss type
depends on $n$ or on the dimension of $\mathcal H$, see also \cite{hb} for related work in terms of the pseudo-spectrum. It has been shown in \cite{beck} and was previously known for ellipses \cite{eiermann} that $\| F_n(A) \|\leq 2=1+v(E)$ provided that $E$ is convex and contains $W(A)$.

Bounding $\| F_n(A) \|$ for a suitable $E\subset \mathbb C$ is also of interest for various other tasks, for instance for spectral inclusions \cite{AES}, Faber hypercyclicity \cite{BaGr}, or the approximate computation of matrix functions \cite{BeRe}. In view of \eqref{est6bis}, we would like $E$ containing $\sigma(A)$ to be well separated from $0$ and to be as small as possible, and thus also allow for non-convex sets.
In the present paper we show the following result.

\begin{theorem}\label{thm}
   Let $0\not\in \sigma(A)$, and consider $E=\{ z \in E_1 ; |z|\geq r \}$ for some $r>0$, with $E_1$ containing $W(A)$ being convex, and $E$ supposed to be simply connected.
   \\ {\bf (a)} If $1/r\geq \| A^{-1} \|$ then $\| F_n(A) \| \leq 1+v(E)$,
   \\ {\bf (b)} If $1/r\geq \max \{ |z|: z \in W(A^{-1}) \}$ then $\| F_n(A) \| \leq 2 v(E)$.
\end{theorem}

The remainder of the paper is organized as follows. In \S 2 we introduce in Lemma~\ref{lem} our new technique of estimating $F_n(A)$ for sets $E$ which are not necessarily convex. This formula is based on an integral formula for Faber polynomials stated in Lemma~\ref{lem_rep_faber}, and used already in \cite{beck}. As a 
by-product, we give a proof of Lemma~\ref{lem_faber}. We then provide a proof of Theorem~\ref{thm}, and discuss in Remark~\ref{rem} possible variations and generalizations of Theorem~\ref{thm}.
\S 3 is devoted to a generalization of the well-known Elman bound \cite{eisen,el} and its recent generalizations in \cite[Theorem~1]{bgt} and \cite[Corollaire~3]{beck} for the convergence of Krylov subspace methods, where $E$ is lens-shaped, allowing to make all constants explicit.

\section{Proof of the main results}
In what follows we will always suppose that the boundary of $E$ is a rectifiable Jordan curve. Also, in order to avoid technical difficulties, in what follows we will assume that $\sigma(A)$ is a subset of the interior of $E$ (the general case following by limit considerations).

We start from a representation of Faber polynomials given already in \cite{beck}, which is a special case of an integral formula given in \cite{crzx} for general polynomials.

\begin{lemma}\label{lem_rep_faber}
  Let $\sigma=\sigma(s)$ be a parametrization of $\partial E$ through arc length, $L$ the length of $\partial E$, and denote by $\nu=\frac{1}{i} \frac{\partial \sigma}{\partial n}$ the unit outer normal of $\partial E$ at $\sigma(s)$ (which by assumption on $E$ exists for almost all $s$). Then for $n\geq 1$
\begin{align}\label{rep}
     F_n(A)&=\int_0^L\Phi(\sigma(s) )^n\mu (s,A)\, ds,
     \quad \mu (s,A):=\frac{1}{2\pi }\big(\nu(\sigma {-}A)^{-1}+\bar\nu(\bar\sigma {-}A^*)^{-1}\big).
 \end{align}
\end{lemma}
\begin{proof}
  Since $(F_n(\sigma )-\Phi(\sigma)^n)(\sigma {-}A)^{-1}$ is analytic outside of $E$ with a behaviour $O(\sigma ^{-2})$ at $\infty$, we have that
  \begin{align*}
       \int_{\partial E}(F_n(\sigma ) {-}\Phi(\sigma)^n)(\sigma {-}A)^{-1}\,d\sigma=0.
  \end{align*}
  Next, we note that
  \begin{align*}
     \int_{\partial E}\Phi(\sigma)^n(\bar\sigma {-}A^*)^{-1}\,d\bar\sigma= \Big(\int_{\partial E}\Phi(\sigma)^{-n}(\sigma {-}A)^{-1}\,d\sigma \Big)^* =0,
  \end{align*}
  since $|\Phi |=1$ on $\partial E$ and $\Phi(\sigma)^{-n}(\sigma {-}A)^{-1}$ is analytic outside of $E$ with a behaviour
  $O(\sigma ^{-(n+1)})$ at $\infty$. Thus, from the Cauchy formula,
  \begin{align*}
        F_n(A)&=\frac{1}{2\pi i}\int_{\partial E}F_n(\sigma )(\sigma {-}A)^{-1}\,d\sigma=
        \frac{1}{2\pi i}\int_{\partial E}\Phi(\sigma)^n(\sigma {-}A)^{-1}\,d\sigma\\
        &=\frac{1}{2\pi i}\int_{\partial E}\Phi(\sigma)^n\Big((\sigma {-}A)^{-1}\,d\sigma -(\bar\sigma {-}A^*)^{-1}\,d\bar\sigma\Big)\\ &=\int_0^L\Phi(\sigma )^n\mu (s,A)\, ds.
   \end{align*}
\end{proof}

\begin{proof}[Proof of Lemma~\ref{lem_faber}]
    The statement of Lemma~\ref{lem_rep_faber} remains valid in the scalar case $A=z$ with $z$ in the interior of $E$. Letting $z$ tend to the boundary $\partial E$ shows the following formula
    implicitly given by K\"ovari and Pommerenke in \cite{koepo, pom}:
    we have for $z\in \partial E$,
    \begin{align*}
        F_n(z)
    &=\Phi (z)^n+\frac{1}{\pi }\int_{\partial E_z}\Phi (\sigma )^n\,d_\sigma \!\arg(\sigma {-}z),
    \end{align*}
    provided that a tangent exists in $z$ (which by assumption on $E$ is true almost everywhere on $\partial E$). Comparing with \eqref{geometric}, it follows that
    $\| F_n - \Phi^n \|_{L^\infty(\partial E)}\leq v(E)$, and thus
    $\| F_n \|_{L^\infty(E)}=\| F_n \|_{L^\infty(\partial E)}\leq 1 + v(E)$, as claimed in Lemma~\ref{lem_faber}. Finally, formula \eqref{value_at_zero} follows from the estimate
 \begin{align*}
        |(F_n(0){-}\Phi^n(0))\Phi (0)|\leq \| F_n - \Phi^n \|_{L^\infty(\partial E)}\leq v(E),
 \end{align*}
obtained by applying the maximum principle to the function $(F_n {-} \Phi^n)\Phi$ being holomorphic outside of $E$ including at $\infty$.
\end{proof}

We are now prepared of stating our main tool for estimating $\| F_n(A)\|$.

\begin{lemma}\label{lem}
    Denote by $\alpha(s)$ the minimum of the (real and compact) spectrum of the self-adjoint operator $\mu (s,A)$ introduced in \eqref{rep}, and by $\alpha_-(s)=\max\{ 0,-\alpha(s)\}$ its negative part. Then for $n \geq 1$,
    we have
    \begin{align*}
        \|F_n(A)\|\leq 2\Big(1+\int_0^L\alpha_-(s)\,ds\Big).
    \end{align*}
\end{lemma}
\begin{proof}
    We first observe that
    \begin{align*}
         \int_0^L \mu(s,A)\, ds = 2 \,\Re \bigl( \frac{1}{2\pi i} \int_{\partial E}  (\sigma - A)^{-1} \, d\sigma \bigr) = 2\, I ,
    \end{align*}
    twice the identity. Taking into account that $\mu(s,A)+\alpha_-(s)$ is positive semi-definite for $s\in [0,L]$ and $|\Phi|=1$ on $\partial E$, we conclude that \begin{align*}
        \Big\|\int_0^L\Phi(\sigma )^n(\mu (s,A){+}\alpha_-(s))\, ds\Big\|&\leq \Big\|\int_0^L(\mu (s,A){+}\alpha_-(s))\, ds\Big\|\\
        &\leq 2+\int_0^L\alpha_-(s)\, ds.
    \end{align*}
    Thus, Lemma~\ref{lem} follows from Lemma~\ref{lem_rep_faber} and the triangular inequality
    \begin{align*}
        \|F_n(A)\|\leq
            \Big\|\int_0^L\Phi(\sigma )^n(\mu (s,A){+}\alpha_-(s))\, ds\Big\|+
            \Big\|\int_0^L\Phi(\sigma )^n\alpha_-(s)\, ds\Big\|.
    \end{align*}
\end{proof}

\begin{figure}[tb]
 \includegraphics[width=17cm,bb=0 560 590 740,clip=true]{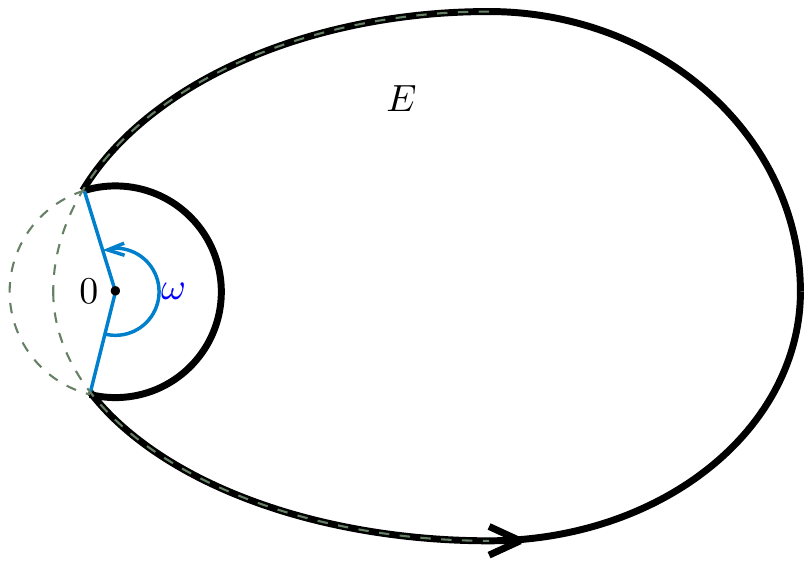}
\caption{Shape of the set $E$ of Theorem~\ref{thm}.}\label{figure1}
\end{figure}

We observe that $\alpha_-(s)=0$ iff $\mu(s,A)$ is positive semi-definite, or, in other words, the numerical range $W(A)$ is a subset of the half plane $\Pi_s:=\{z\in\C\,; \Re\overline{\nu(s)}(\sigma(s) {-}z)\geq 0\}$, with boundary being tangent to $\partial E$ in $\sigma(s)$. Thus for convex $E$ containing $W(A)$ we deduce from Lemma~\ref{lem} the bound $\| F_n(A) \| \leq 2$ mentioned before.

\begin{proof}[Proof of Theorem~\ref{thm}.] According to Fig.~\ref{figure1}, we have to distinguish two cases: if $\sigma(s) \subset \partial E \cap E_1$, the convex part of $\partial E$, then by assumption $W(A) \subset E_1 \subset \Pi_s$ and thus $\alpha_-(s)=0$.

It remains to analyze the circular part of the boundary which can be parametrized by $\sigma(s)=re^{i\theta}$, with $\theta$ decreasing from $\theta_1$ to $\theta_1-\omega$, $\omega$ being the opening angle as introduced in Fig.~\ref{figure1}. Then $ds=-r\, d\theta$ and $\nu(s)=-e^{i\theta}$. Consider the operator $B:=re^{i\theta} A^{-1}$.

We first consider the case {\bf (a)} in which $\| B \| \leq 1$, implying that
\begin{align*}
   2\pi r\mu (s,A)+1&=1-re^{i\theta }(re^{i\theta}{-}A)^{-1}-re^{-i\theta }(re^{-i\theta}{-}A^*)^{-1}\\
   &=1+B(I{-}B)^{-1}+B^*(I{-}B^*)^{-1} = \Re \bigl( (I+B)(I-B)^{-1} ) \geq 0.
\end{align*}
Hence, on this part of the boundary, $\alpha_-(s)\leq 1/(2\pi r)$.
It follows from Lemma~\ref{lem}, that $\|F_n(A)\|\leq 2 -\int_{\theta_1}^{\theta_1{-}\omega}d\theta /\pi =2 {+}\omega/\pi$.

In case {\bf (b)} we have the weaker assumption $W(B)\subset \mathbb D$ and thus $2 - B - B^*\geq 0$, implying that
\begin{align*}
   2\pi r\mu (s,A)+2&=2-re^{i\theta}(re^{i\theta}{-}A)^{-1}
   -
   re^{-i\theta}(re^{-i\theta}{-}A^*)^{-1}\\
   &=2+B(I{-}B)^{-1}+B^*(I{-}B^*)^{-1}\\
   &=(I{-}B)^{-1}(2I{-}B{-}B^*)(I{-}B^*)^{-1}\geq 0.
\end{align*}
Thus as before we conclude from Lemma~\ref{lem} that $\|F_n(A)\|\leq 2 {+}2 \omega/\pi$.\medskip

It remains to show that $v(E)=1{+}\omega/\pi$. For that, we note that
$\frac1\pi d_\sigma (\arg(\sigma(s) {-}z))=\mu(s,z)\,ds$, whence
\begin{align*}
v(E)=1+2\,\supess\{ \int_{\partial E_z}\mu_-( s,z)\,ds\,; z\in \partial E\}.
\end{align*}
For $z\in\partial E$, it holds $\mu_-( s,z)=0$ if $\sigma (s)\in \partial E\cap E_1$,
$\mu_-( s,z)\leq 1/2\pi r$ if $\sigma (s)$ belongs to the circular part
and $\mu_-( s,z)= 1/2\pi r$ if $\sigma (s)$ and $z$ belong to the circular part.
This shows that $\int_{\partial E_z}\mu_-( s,z)\,ds \leq 2\,\omega/\pi $, with equality if $z$ belongs to the circular part.
\end{proof}

\begin{figure}[tb]
 \includegraphics[width=17cm,viewport=0 560 590 740,clip=true]{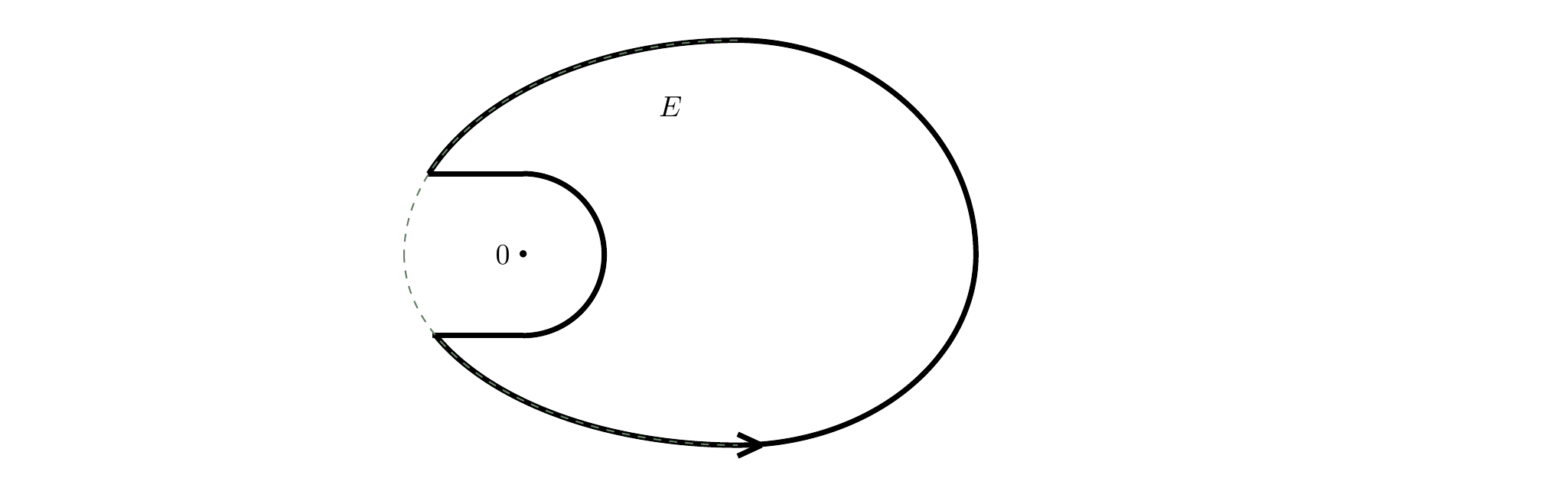}\caption{Shape of the set $E$ of Remark~\ref{rem}(b).}\label{figure2}
\end{figure}
Let us discuss variations and generalizations of Theorem~\ref{thm}.

\begin{remark}\label{rem}{\rm
  {\bf (a)} Theorem~\ref{thm}(b) remains valid if we replace$|z|\geq r$ by $|z{-}c|\geq r$ in the definition of $E$,
  and if we assume $1/r\geq \sup\{|z|\,; z\in W((A{-c)}^{-1})\}$.
  The reader will not have difficulty to generalize this to simply connected domains of the form $E=\{z\in E_1\,; |z-c_1|\geq r_1,\dots,  |z-c_k|\geq r_k\}$. \\
  {\bf (b)} Other variations are possible. For instance, if we consider $E_1$ a compact convex set such that $W(A)\subset E_1$, then $E=\{z\in E_1\,; |z{+}x|\geq r$  for all $x\geq 0\}$ is simply connected, see Fig.~\ref{figure2}. If we suppose in addition that
  $1/r\geq \sup\{|z|\,; z\in W((A{+}x)^{-1})\}$ for all $x\geq 0$, then we have $\alpha_-(s)=0$ on $\partial E\cap E_1$ and $\alpha_-(s)\leq 1/(\pi r)$ on the remaining part of the boundary. It seems however that in general there is no obvious link between the resulting bound for $\| F_n(A)\|$ and $v(E)$.
}\end{remark}

\section{An application to Krylov subspace methods}

An interesting application of the estimation of the quantity $\delta_n(A)$ defined in \eqref{but} concerns the study of convergence of Krylov subspace methods such as FOM, GMRES, BiCG, QMR,\dots, see for instance \cite{greenbaum} and the references therein. These methods are very popular for solving linear systems $Ax=b$ of large dimension $N$. Here $A$ is a $N\times N$ matrix with complex entries and $b\in \C^N$ a given vector. At the step $n$, one constructs an approximation $x_n$ of the solution $x$ which belongs to the Krylov subspace $K_n=$ Span$\{b,Ab,\dots,A^{n-1}b\}$. The above-mentioned Krylov subspace methods differ in two ways, namely the choice of the basis of $K_n$,
and the choice of the linear combination on this basis. But, they all allow for an error estimate
of the following type
\begin{align*}
\|x-x_n\|\leq \|A^{-1}\Pi _n\| \min_{p(0)=1,\ p\in \PP_n}\|p(A)b\| \leq \delta_n(A)\|A^{-1}\Pi _n\|  \|b\|.
\end{align*}
Here, $I{-}\Pi _n$ is a projector on the Krylov subspace $A\,K_n$, the orthogonal projector in the GMRES case.

We want to make our bounds \eqref{est6bis} together with Theorem~\ref{thm} more explicit by choosing a particular non-convex lens-shaped set $E$, which allows us to express the occuring constants $\| F_n(A)\|$ and $\gamma$ in terms of angles related to $A$.

\begin{figure}[tb]
 \includegraphics[width=17cm,viewport=0 560 590 740,clip=true]{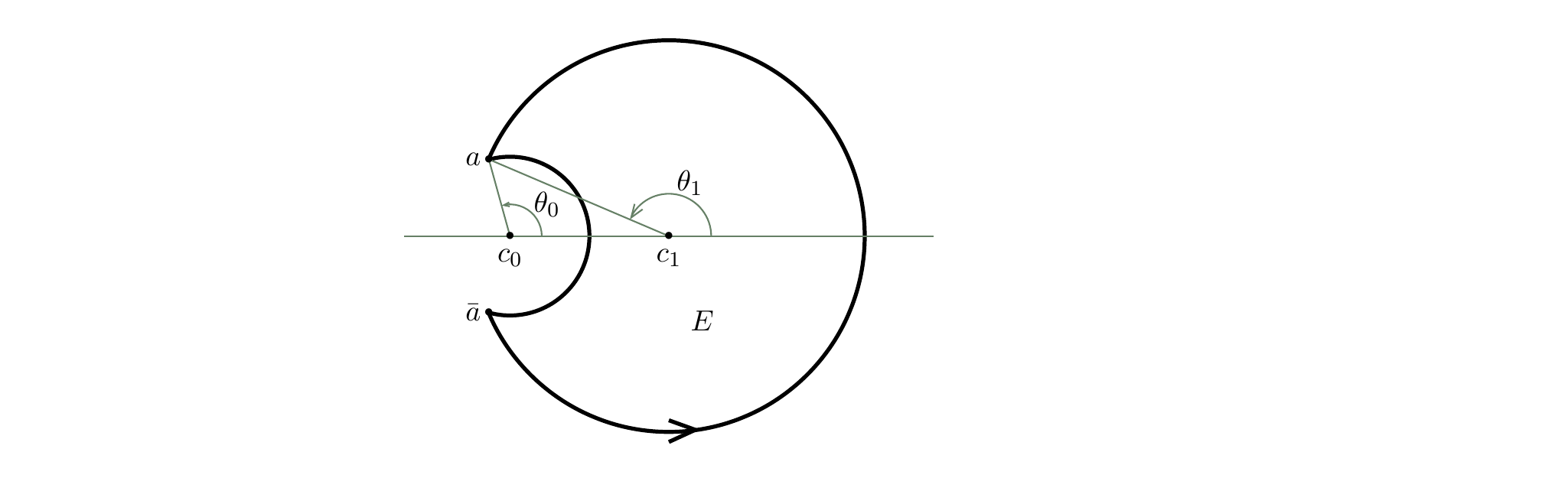}
\caption{Shape of the set $E$ of Corollary~\ref{cor} and angles $0<\theta_0<\theta_1<\pi$.}\label{figure3}
\end{figure}

\begin{corollary}\label{cor}
   Consider $E=\{z\in \C\,; |z{-}c_1|\leq r_1$ and $|z{-}c_0|\geq r_0\}$
   with $W(A)$ being contained in some ball of radius $r_1>0$ centered at $c_1\in \mathbb R$, and $W((A{-}c_0)^{-1})$ being contained in some ball of radius $1/r_0>0$ centered at $0$, where  $c_0\in \mathbb R$.
   We suppose that
   \begin{align*}
      c_0{-}r_0<c_1{-}r_1<c_0{+}r_0<c_1{+}r_1 \quad \mbox{and}
      \quad 0<c_0{+}r_0
   \end{align*}
   such that $E$ is (lens-shaped and) simply connected as in Fig.~\ref{figure3} and does not contain $0$.

   For the endpoint $a$ of the circular arcs composing $\partial E$, we introduce
   the angles $\theta_0=\arg(a{-}c_0)$, $\theta_1=\arg(a{-}c_1), \arg(a)\in (0,\pi)$, Then
   \begin{align*}
\| F_n(A)\| \leq 2{+}4\theta _0/\pi \quad\hbox{and} \quad \gamma =\frac{\sin(\frac{\pi \arg a}{2\pi -\theta _1+\theta_0})}{\sin(\frac{\pi (\pi +\theta _0-\arg a)}{2\pi -\theta _1+\theta_0})}.
   \end{align*}
\end{corollary}
\begin{proof}
 The upper bound for $\| F_n(A)\|$ follows from Theorem~\ref{thm} and Remark~\ref{rem}(a) by recalling that $v(E)=1+2\theta_0/\pi$. For showing the claimed formula for $\gamma$, let us construct explicitly the corresponding map $\Phi$. We consider
 \begin{align*}
\varphi _1(z):=e^{-i\frac{\theta _1+\theta_0}{2}}\frac{z-a}{z-\bar a}\ \ \hbox{and}\ \
S=\{\rho \,e^{i\mu }\,; \rho >0,\ -\pi{+}
 \tfrac{\theta _1-\theta_0}{2}<\mu <\pi {-}\tfrac{\theta _1-\theta_0}{2}\},
\end{align*}
where we notice that $c_0<c_1$ by assumption, and thus $0<\theta_1-\theta_0<\pi$.
It is easily seen that $\varphi _1$ maps the exterior of $E$ onto the sector $S$. For $z$ exterior to $E$, we can define in the canonical way $\varphi _2(z):=\varphi _1(z)^{\pi /(2\pi-\theta _1+\theta _0)}$, so that
 $\varphi _2$ maps the exterior of $E$ onto the half-plane $\Re z>0$. Finally, we define
 \begin{align*}
 \Phi (z):=\frac{\overline{\varphi _2(\infty)}+\varphi_2(z)}{\varphi _2(\infty)-\varphi_2(z)}\,e^{i\mu },\quad\hbox{$\mu\in\mathbb R $ being chosen such that }\Phi '(\infty)>0,
 \end{align*}
 and note that $\Phi $ maps the exterior of $E$ onto the exterior of the unit disk and $\Phi (\infty)=\infty$. We have $\varphi _2(\infty)=e^{-i\alpha}$ and $\varphi _2(0)=e^{i\beta }$ with $\alpha=\frac{\pi }{2\pi -\theta _1+\theta _0}\frac{\theta _1+\theta_2}{2}$ and $\beta =2\frac{\pi }{2\pi -\theta _1+\theta _0}\arg a-\alpha$. Therefore
  \begin{align*}
 \gamma =\frac{1}{|\Phi (0)|}=\Big|\frac{e^{-i\alpha}-e^{i\beta }}{e^{i\alpha}+e^{i\beta }}\Big|
 =\Big|\frac{\sin\frac{\alpha+\beta}{2} }{\cos\frac{\alpha-\beta}{2} }\Big|=
 \frac{\sin(\frac{\pi \arg a}{2\pi -\theta _1+\theta_0})}{\sin(\frac{\pi (\pi +\theta _0-\arg a)}{2\pi -\theta _1+\theta_0})}.
 \end{align*}
\end{proof}

As an illustration of Corollary~\ref{cor}, consider the situation of \cite[Corollaire~3]{beck} where $c_1=0$, $r_1=\max\{ |z|\,; z \in W(A) \}$,
and $\Re(A) \geq \alpha=\cos(\beta) \, r_1>0$ for some $\beta\in (0,\pi/2)$.
Then for all $c_0<\alpha$ we find that
\begin{align*}
    \| (A-c_0)^{-1} \| \leq \frac{1}{\mbox{dist}(c_0,W(A))}\leq \frac{1}{\alpha-c_0}=:\frac{1}{r_0} ,
\end{align*}
and, for $c_0\to -\infty$, we see that $\theta_0\to 0$, and $\arg a=\theta_1\to \beta$. Hence the generalization \cite[Corollaire~3]{beck} of Elman's bound \cite{eisen,el}
with $\gamma
 ={\sin(\frac{\pi \beta}{2\pi -\beta})}
 /
 {\sin(\frac{\pi(\pi-\beta)}{2\pi -\beta})}=2\, {\sin(\frac{\pi \beta/2}{2\pi -\beta})}$ follows as a limiting case from Corollary~\ref{cor}.

\bigskip

{\bf Acknowledgements.}
The authors gratefully acknowledge valuable discussions with Catalin Badea.

\bigskip


\bibliographystyle{amsplain}

\bigskip

\noindent Laboratoire Paul Painlev\' e, UMR CNRS no. 8524,\\
Universit\'e de Lille 1, 59655 Villeneuve d'Ascq Cedex, France
\\
{\tt Bernhard.Beckermann@math.univ-lille1.fr}\\

\noindent Institut de Recherche Math\'ematique de Rennes, UMR 6625
au CNRS,\\ Universit\'e de Rennes 1, Campus de Beaulieu, 35042
RENNES Cedex, France\\ {\tt michel.crouzeix@univ-rennes1.fr}



\end{document}